\documentclass[11pt,reqno]{amsart}

\usepackage[margin=1in]{geometry}
\usepackage{amsmath,amssymb,amsthm,mathtools}
\usepackage{stmaryrd}
\usepackage{enumitem}
\usepackage{extarrows}
\usepackage{xcolor}
\usepackage{tikz}
\usepackage{hyperref}

\hypersetup{
  colorlinks=true,
  linkcolor=blue,
  citecolor=blue,
  urlcolor=blue
}

\allowdisplaybreaks
\numberwithin{equation}{section}

\theoremstyle{plain}
\newtheorem{theorem}{Theorem}[section]
\newtheorem{lemma}[theorem]{Lemma}
\newtheorem{proposition}[theorem]{Proposition}

\newtheorem{question}[theorem]{Question}
\newtheorem{conjecture}[theorem]{Conjecture}
\newtheorem*{recalledtheorem}{Theorem}

\theoremstyle{definition}

\newcommand{\R}{\mathbb{R}}
\newcommand{\N}{\mathbb{N}}
\newcommand{\Z}{\mathbb{Z}}

\newcommand{\Sc}{\operatorname{Sc}}
\newcommand{\dd}{\mathop{}\!d}

\newcommand{\Vol}{\operatorname{vol}}
\newcommand{\Ric}{\operatorname{Ric}}
\newcommand{\secop}{\operatorname{sec}}
\newcommand{\hyp}{\mathrm{hyp}}
\newcommand{\sph}{\mathbb{S}}

\title[A sharp hyperbolic volume bound]{A sharp hyperbolic volume bound for hypersurfaces in \(M^3\times \sph^1\)}
\author{Lizhi Chen}
\address[Lizhi Chen]{School of Mathematics and Statistics, Lanzhou University}
\email{lizhi.chen.math@gmail.com}
\thanks{The first author is partially supported by NSFC 12271225.}

\author{Kuntao Jin}
\address[Kuntao Jin]{Department of Mathematical Sciences, Tsinghua University}
\email{jkt25@mails.tsinghua.edu.cn}

\begin{document}

\begin{abstract}
Let $(M^3, g_{\mathrm{hyp}})$ be a closed oriented hyperbolic three-manifold normalized so that $\operatorname{sec}_{g_{\mathrm{hyp}}} \equiv -1$. We prove a sharp lower bound for the volume of hypersurfaces in $M^3 \times \mathbb{S}^1$ representing the slice class $[M^3 \times \{ \mathrm{pt} \}] \in H_3(M^3 \times \mathbb{S}^1; \mathbb{Z})$, and we classify the equality case. If $g$ is a smooth Riemannian metric on $M^3 \times \mathbb{S}^1$ with the scalar curvature $\operatorname{Sc}_g \geq -6$, then every closed embedded hypersurface $\Sigma$ representing the slice class $[M^3\times\{\mathrm{pt}\}]$ satisfies $\operatorname{vol}_g(\Sigma) \geq \operatorname{vol}_{g_{\mathrm{hyp}}}(M^3)$. The bound is attained by the product metric $g_{\mathrm{hyp}}+h$, with $h$ any metric on $\mathbb{S}^1$. Conversely, if equality holds for some $\Sigma$, then up to a diffeomorphism preserving the slice class, $g=g_{\mathrm{hyp}}+h$ and $\Sigma = M^3 \times \{\mathrm{pt}\}$. 
\end{abstract}

\maketitle

\section{Introduction}

Scalar curvature can be read infinitesimally from the volume growth of small
geodesic balls.  Indeed, if \((X^n,g)\) is a Riemannian manifold, then as
\(r\to0\),
\[
  \Vol_g(B_g(p,r))
  =
  \omega_n r^n
  \left(
    1-\frac{\Sc_g(p)}{6(n+2)}r^2+O(r^4)
  \right),
  \qquad r\to0,
\]
where \(\omega_n\) is the Euclidean unit \(n\)-ball volume. Thus a lower bound for \(\Sc_g\) gives, to second order, an upper bound for the volume of small geodesic balls. Motivated by this relation between scalar curvature and ball
volumes, Gromov proposed a finite-scale analogue, now called macroscopic scalar curvature.  This notion replaces the infinitesimal comparison above by a fixed-scale comparison on the universal cover (see \cite{GuthMetaphors} and
\cite{GuthVolumesBalls}).  More precisely, fix \(r>0\), let
\((\widetilde X,\widetilde g)\) be the Riemannian universal cover, and write
\[
  \widetilde V(\widetilde p,r)
  =
  \Vol_{\widetilde g}\bigl(B_{\widetilde g}(\widetilde p,r)\bigr).
\]
For a real number \(S\), let \(\widetilde V_S(r)\) be the volume of the radius \(r\) ball in the simply connected \(n\)-dimensional space form of scalar curvature \(S\).  Following Guth \cite{GuthMetaphors,GuthVolumesBalls},
the macroscopic scalar curvature at
scale \(r\) and at \(p\) is the number \(S\) determined by
\[
  \widetilde V(\widetilde p,r)=\widetilde V_S(r),
\]
where \(\widetilde p\) is any lift of \(p\). Therefore a lower bound by \(S\) at scale \(r\) is exactly the fixed-scale volume comparison
\[
  \Vol_{\widetilde g}\bigl(B_{\widetilde g}(\widetilde p,r)\bigr)
  \leq
  \widetilde V_S(r).
\]
At scale \(1\), this condition is a uniform upper bound for the volumes of
unit balls in \((\widetilde X,\widetilde g)\).  In terms of macroscopic scalar
curvature, the area theorem of Alpert--Funano can be stated as follows (see
\cite[Theorem~1.2]{AlpertFunano}).

\begin{recalledtheorem}
Let \((\Sigma^n,g_{\hyp})\) be an $n$-dimensional closed hyperbolic manifold, and set
\(X=\Sigma\times\sph^1\).  For every \(S\in\R\), there exists a constant
\(c=c(n,S)>0\) with the following property.  If \(g\) is a Riemannian metric
on \(X\) whose macroscopic scalar curvature at scale \(1\) is everywhere
bounded below by \(S\), then every smooth embedded hypersurface \(Z\subset X\)
representing the slice class
\([\Sigma\times\ast]\in H_n(X;\mathbb Z_2)\) satisfies
\[
  \Vol_g(Z)\geq c\,\Vol_{g_{\hyp}}(\Sigma).
\]
\end{recalledtheorem}

The Alpert--Funano theorem naturally suggests replacing its macroscopic hypothesis by the pointwise scalar curvature lower bound of the hyperbolic product. This leads to the following sharp volume and rigidity question.

\begin{question}[Scalar curvature analogue]
Let \((M^n,g_{\hyp})\) be a closed oriented hyperbolic \(n\)-manifold with \(\secop_{g_{\hyp}}\equiv -1\), and let \(g\) be a Riemannian metric on \(M^n\times\sph^1\) satisfying
\[
  \Sc_g\geq -n(n-1).
\]
Does every volume-minimizing representative \(Z\) of the slice class
\( [M^n\times \{pt\}] \in H_n(M^n \times \sph^1; \Z) \) satisfy
\[
  \Vol_g(Z)\geq \Vol_{g_{\hyp}}(M)?
\]
Moreover, does equality hold if and only if, for some metric \(h\) on
\(\sph^1\),
\[
  (M^n\times\sph^1,g)
  \quad\text{is isometric to}\quad
  (M^n,g_{\hyp})\times(\sph^1,h)?
\]
\end{question}

\bigskip

Our main result proves this scalar curvature analogue in dimension three.

\begin{theorem}\label{thm:intro-main}
Let \((M^3,g_{\hyp})\) be a closed oriented hyperbolic three-manifold with \(\secop_{g_{\hyp}}\equiv -1\), and let \(g\) be a smooth Riemannian metric on \(M^3\times\sph^1\) satisfying
\[
  \Sc_g\geq -6.
\]
Then every closed embedded hypersurface \(\Sigma\) representing the slice class \( [M^3 \times \{ pt\}] \in H_3(M^3 \times \sph^1; \Z) \) satisfies
\[
  \Vol_g(\Sigma)\geq \Vol_{g_{\hyp}}(M).
\]
Moreover, equality holds if and only if, up to a diffeomorphism preserving the slice class, for some flat metric \(h\) on \(\sph^1\), \(  (M^3\times\sph^1,g) \) is isometric to \( (M^3,g_{\mathrm{\hyp}})\times(\sph^1,h) \).
\end{theorem}

\bigskip

The proof uses a sharp three-dimensional Yamabe estimate for manifolds admitting a nonzero-degree map to a closed hyperbolic three-manifold.  The stability inequality for a minimizing hypersurface, together with the Gauss equation, give a lower bound for the Yamabe constant of the induced conformal class.  On a component mapping to \(M\) with nonzero degree, the degree Yamabe estimate gives the opposite inequality. Comparing these two inequalities gives the volume lower bound. The equality case then forces equality in the stability and Yamabe estimates; the remaining step is the passage from local splitting near the minimizing hypersurface to global splitting. The higher-dimensional analogue is discussed in Section~\ref{sec:higher-dimensional}.

\bigskip

\noindent \textbf{On the use of AI.} In preparing this work, the authors used ChatGPT (GPT-5.5 and GPT-5.6 Sol) to explore possible proof strategies and test their own arguments, and to improve the exposition and language. All mathematical statements and proofs were formulated by the authors, who independently verified every argument and take full responsibility for this paper.

\bigskip

\section{Background and scalar curvature comparison}

We recall two standard notions used in the proof: the Yamabe invariant and simplicial volume. The first is a conformal invariant defined from scalar curvature, while the second is a topological invariant closely related to hyperbolic volume. In dimension three these two invariants give the sharp degree estimate used below.

Let \(X^n\) be a closed smooth manifold, \(n\geq3\).  For a conformal class \([h]\), the Yamabe constant is
\[
  Y(X,[h]) =  \inf_{\varphi\in C^\infty(X),\ \varphi \not\equiv 0} \frac{  \displaystyle \int_X
  \left( \frac{4(n-1)}{n-2}|\nabla\varphi|^2+\Sc_h\varphi^2
  \right) \dd \mu_h }{ \displaystyle
  \left(\int_X|\varphi|^{\frac{2n}{n-2}} \dd\mu_h \right)^{\frac{n-2}{n}} }.
\]
Moreover, the Yamabe invariant of \(X\) is
\[
  \sigma(X)=\sup_{[h]}Y(X,[h]),
\]
where the supremum is taken over all conformal classes on \(X\). Thus \(Y(X,[h])\) is the scalar curvature invariant of one conformal class, while \(\sigma(X)\) is obtained by optimizing over all conformal classes.

Schoen--Yau discuss the problem of determining the optimal Yamabe constant of a closed hyperbolic manifold.  In the normalization used here, the expected answer can be written in the following form.

\begin{conjecture}[Schoen's Yamabe invariant conjecture, {\cite{SchoenVariational}}]\label{conj:schoen-yamabe}
Let \((M^n,g_{\hyp})\) be a closed oriented $n$-dimensional hyperbolic manifold normalized by \(\secop_{g_{\hyp}}\equiv -1\).  Then the hyperbolic metric realizes the Yamabe invariant, namely
\[
  \sigma(M)  =  -n(n-1)\Vol_{g_{\hyp}}(M)^{2/n}.
\]
\end{conjecture}

Schoen's hyperbolic volume conjecture is the corresponding scalar curvature form of this problem: it asks whether scalar curvature alone recovers the hyperbolic volume of \(M\). For the equivalence between Conjecture~\ref{conj:schoen-yamabe} and the volume comparison below, see
Yuan \cite{YuanVolumeComparison}.

\begin{conjecture}[Schoen's hyperbolic volume conjecture, {\cite{SchoenVariational}}]\label{conj:schoen-volume}
Let \((M^n,g_{\hyp})\) be a closed oriented $n$-dimensional hyperbolic manifold normalized by \(\secop_{g_{\hyp}}\equiv -1\).  If \(g\) is another Riemannian metric on \(M\) satisfying
\[
  \Sc_g\geq -n(n-1),
\]
then
\[
  \Vol_g(M)\geq \Vol_{g_{\hyp}}(M).
\]
Moreover, equality holds if and only if \(g\) is hyperbolic.
\end{conjecture}

This is the scalar curvature analogue of the minimal entropy theorem of Besson--Courtois--Gallot (see \cite{BCGMinimalEntropy}).  A macroscopic version was studied by Balacheff--Karam \cite{BalacheffKaram}.

The topological counterpart is simplicial volume.  Let \(X\) be a closed oriented \(n\)-manifold.  For a real singular \(n\)-chain
\[
  c=\sum_{i=1}^N a_i\sigma_i\in C_n(X;\mathbb R),
\]
set
\[
  |c|_1=\sum_{i=1}^N |a_i|.
\]
The simplicial volume of \(X\) is the \(\ell^1\)-seminorm of its real fundamental class:
\[
  \|X\|  =  \inf\bigl\{ |c|_1:   c\in C_n(X;\mathbb R),\ \partial c=0,\ [c]=[X]_{\mathbb R}  \bigr\}.
\]
Here \([X]_{\mathbb R}\in H_n(X;\mathbb R)\) denotes the real fundamental class. For a closed oriented hyperbolic \(n\)-manifold \((M^n,g_{\hyp})\), normalized by \(\secop_{g_{\hyp}}\equiv -1\), by the Gromov proportionality principle (see \cite{GromovVolumeBound} and \cite{ThurstonGeometry}), one has
\[
  \|M\|=\frac{\Vol_{g_{\hyp}}(M)}{v_n},
\]
where \(v_n\) is the volume of the regular ideal \(n\)-simplex in
\(\mathbb H^n\).  In this sense, the hyperbolic volume is also a topological quantity. Simplicial volume is a homotopy invariant of closed oriented
manifolds, and Gromov's degree monotonicity \cite[p.~8]{GromovVolumeBound} says that if \(f:X\to M\) has
degree \(d\), then
\[
  \|X\|\geq |d|\,\|M\|.
\]

Our theorem is a codimension-one version of this comparison for \(M^3\times\sph^1\). Instead of estimating the volume of the whole ambient manifold, we estimate the least volume in the slice class.  The analytic bridge is the Yamabe invariant.

We shall use the following three-dimensional consequence of known results.
It is not a separate theorem of a single paper in precisely this form.  Rather,
it follows by combining the computation of the Yamabe invariant of closed
three-manifolds via geometrization (see \cite{AndersonCanonical} and
\cite{AILPerelmanYamabe}) with the Gromov proportionality principle for the relation between simplicial volume and hyperbolic volume.

\begin{lemma}\label{thm:yamabe-degree}
Let \((M^3,g_{\hyp})\) be a closed oriented hyperbolic three-manifold normalized by
\(\secop_{g_{\hyp}}\equiv -1\).  Let \(X\) be a closed oriented
three-manifold.  If \(f:X\to M\) is a smooth map of degree \(d\neq0\), then
\[
  \sigma(X)
  \leq
  -6\left(|d|\Vol_{g_{\hyp}}(M)\right)^{2/3}.
\]
\end{lemma}

\begin{proof}
Let \(V_{\hyp}(X)\) be the sum of the volumes of the hyperbolic pieces in the geometric decomposition of \(X\). By Gromov additivity theorem (see {\cite[p.~58]{GromovVolumeBound}}, also see {\cite[Section 7.4]{Frigerio2017}}) and Gromov proportionality principle,  
\[
  V_{\hyp}(X)=v_3\|X\|.
\]
Here \(v_3\) is the volume of a regular ideal tetrahedron in \(\mathbb H^3\). For the hyperbolic manifold \(M\),
\[
  \Vol_{g_{\hyp}}(M)=v_3\|M\|.
\]
By degree monotonicity of simplicial volume (see {\cite[p.~8]{GromovVolumeBound}}),
\[
  V_{\hyp}(X)  =  v_3\|X\| \geq |d|v_3\|M\|
  = |d|\Vol_{g_{\hyp}}(M).
\]
In particular \(V_{\hyp}(X)>0\). Then $X$ admits no metrics of positive scalar curvature, thus $\sigma(X) \leq 0$. The computation of the Yamabe invariant of closed three-manifolds gives
\[
  \sigma(X)=-6V_{\hyp}(X)^{2/3}
\]
in this case (see \cite{AndersonCanonical} and
\cite{AILPerelmanYamabe}).  Therefore
\[
  \sigma(X)  \leq  -6\left(|d|\Vol_{g_{\hyp}}(M)\right)^{2/3}.
\]
\end{proof}

\bigskip

\section{Proof of the Main Theorem}
In this section, we prove the main theorem (Theorem~\ref{thm:intro-main}).
\begin{recalledtheorem}
Let \((M^3,g_{\hyp})\) be a closed oriented hyperbolic three-manifold with \(\secop_{g_{\hyp}}\equiv -1\), and let \(g\) be a smooth Riemannian metric on \(M^3\times\sph^1\) satisfying \(\Sc_g\geq -6\).  Then every closed embedded hypersurface \(\Sigma\) representing the slice class \( [M^3 \times \{ pt\}] \in H_3(M^3\times \sph^1; \Z) \) satisfies
\[
  \Vol_g(\Sigma)\geq \Vol_{g_{\hyp}}(M).
\]
Moreover, equality holds if and only if, up to a diffeomorphism preserving the slice class, \(g\) splits as \(g_{\hyp}+h_{\sph^1}\).
\end{recalledtheorem}

\begin{proof}
Put \(N=M^3\times\sph^1\), and let
\[
  \alpha=[M^3\times\{\mathrm{pt}\}]\in H_3(M^3\times \sph^1;\Z).
\]
We prove the slightly stronger current version of the lower bound. Choose an integral current \(T\) minimizing mass in the class \(\alpha\).  By regularity for codimension-one mass-minimizing currents in ambient dimension \(4\), \(T\) is represented by a smooth embedded oriented stable minimal hypersurface, possibly disconnected and with integer multiplicities (see \cite{FedererFleming} and \cite{SimonGMT}).  Thus
\[
  T=\sum_j m_j \llbracket \Sigma_j \rrbracket ,
\]
where each \(\Sigma_j\) is a connected smooth closed oriented stable minimal hypersurface and \(m_j\in\N\).  Its mass is
\[
  \mathbf M(T)=\sum_j m_j\Vol_g(\Sigma_j).
\]

We first prove the estimate supplied by stability.  Let \(\Gamma\) be one of the components \(\Sigma_j\), and let \(\nu\) be a unit normal along
\(\Gamma\).  We use the Schoen--Yau stable hypersurface argument: combine the second variation inequality with the Gauss equation (see \cite{SchoenYauLectures}).  For every \(\varphi\in C^\infty(\Gamma)\), stability gives
\[
  \int_\Gamma
  \left(
  |\nabla \varphi|^2
  -
  \bigl(\Ric_N(\nu,\nu)+|A|^2\bigr)\varphi^2
  \right)
  \geq 0.
\]
The Gauss equation for a minimal hypersurface is
\[
  \Sc_\Gamma  =  \Sc_N  -  2\Ric_N(\nu,\nu)  -  |A|^2.
\]
Hence
\[
  \Ric_N(\nu,\nu)+|A|^2  =  \frac12\left(\Sc_N-\Sc_\Gamma+|A|^2\right).
\]
Substitution in the stability inequality gives
\[
  \int_\Gamma \left(  2|\nabla\varphi|^2+\Sc_\Gamma\varphi^2  \right)  \geq  \int_\Gamma  \left(  \Sc_N+|A|^2  \right)\varphi^2.
\]
Since \(\Sc_N=\Sc_g\geq -6\), it follows that
\[
  \int_\Gamma  \left(  2|\nabla\varphi|^2+\Sc_\Gamma\varphi^2  \right)  \geq  -6\int_\Gamma \varphi^2.
\]

In dimension \(3\), the Yamabe quotient of the induced metric is
\[
  Q_\Gamma(\varphi)
  =
  \frac{
  \displaystyle
  \int_\Gamma
  \left(
  8|\nabla\varphi|^2+\Sc_\Gamma\varphi^2
  \right)
  }{
  \displaystyle
  \left(\int_\Gamma |\varphi|^6\right)^{1/3}
  }.
\]
Because
\[
  8|\nabla\varphi|^2+\Sc_\Gamma\varphi^2  =  \left(  2|\nabla\varphi|^2+\Sc_\Gamma\varphi^2  \right)  +  6|\nabla\varphi|^2,
\]
we have
\[
  \int_\Gamma
  \left(
  8|\nabla\varphi|^2+\Sc_\Gamma\varphi^2
  \right)
  \geq
  -6\int_\Gamma \varphi^2.
\]
H\"older's inequality gives
\[
  \int_\Gamma \varphi^2
  \leq
  \Vol_g(\Gamma)^{2/3}
  \left(\int_\Gamma |\varphi|^6\right)^{1/3}.
\]
Therefore
\[
  Q_\Gamma(\varphi)\geq -6\Vol_g(\Gamma)^{2/3}.
\]
Taking the infimum over all nonzero \(\varphi\), we obtain
\[
  Y(\Gamma,[g_\Gamma])
  \geq
  -6\Vol_g(\Gamma)^{2/3}.
\]
This estimate holds for every component \(\Gamma=\Sigma_j\).

Now denote by
\[
  p\colon M^3\times \sph^1\longrightarrow M^3
\]
the projection. 
Since \(T\) represents the homology class
\([M^3\times\{\mathrm{pt}\}]\), pushing forward by \(p\) gives
\[
   [p_\#T] = [M^3].
\]
Equivalently,
\[
  \sum_j m_j\deg(p|_{\Sigma_j})=1.
\]
Hence some component has nonzero degree over \(M\).  Fix such a component and
write it as \(\Sigma_{j_0}\).  Set
\[
  d=\deg(p|_{\Sigma_{j_0}})\neq0.
\]
Applying Lemma~\ref{thm:yamabe-degree} to
\[
  p|_{\Sigma_{j_0}}:\Sigma_{j_0}\to M
\]
and using \(Y(\Sigma_{j_0},[g_{\Sigma_{j_0}}])\leq\sigma(\Sigma_{j_0})\)
gives
\[
  Y(\Sigma_{j_0},[g_{\Sigma_{j_0}}])
  \leq
  -6\left(|d|\Vol_{g_{\hyp}}(M)\right)^{2/3}.
\]
Combining this with the stability lower bound gives
\[
  -6\Vol_g(\Sigma_{j_0})^{2/3}
  \leq
  Y(\Sigma_{j_0},[g_{\Sigma_{j_0}}])
  \leq
  -6\left(|d|\Vol_{g_{\hyp}}(M)\right)^{2/3}.
\]
Hence
\[
  \Vol_g(\Sigma_{j_0})^{2/3}
  \geq
  \left(|d|\Vol_{g_{\hyp}}(M)\right)^{2/3}.
\]
Raising both sides to the power \(3/2\), we get
\[
  \Vol_g(\Sigma_{j_0})
  \geq
  |d|\Vol_{g_{\hyp}}(M)
  \geq
  \Vol_{g_{\hyp}}(M).
\]
Finally,
\[
  \mathbf M(T)
  =
  \sum_j m_j\Vol_g(\Sigma_j)
  \geq
  \Vol_g(\Sigma_{j_0}),
\]
so
\[
  \mathbf M(T)
  \geq
  \Vol_{g_{\hyp}}(M).
\]
Thus the mass minimum of the slice class is at least
\(\Vol_{g_{\hyp}}(M)\).  In particular every smooth volume minimizing $2$-sided hypersurface in the slice class $[M^3 \times \{ \mathrm{pt} \}]$ satisfies the stated volume lower bound.

The product metric
\[
  g_0=g_{\hyp}+h_{\sph^1}
\]
on \(M^3\times\sph^1\), where \(h_{\sph^1}\) is any metric on \(\sph^1\),
shows that the constant is sharp. Since the scalar curvature of \(g_{\hyp}\) is \(-6\),
\[
  \Sc_{g_0}=-6.
\]
Each slice \(M^3\times\{ \mathrm{pt} \}\) is dual to the \(\sph^1\)-factor and has volume
\[
  \Vol_{g_0}(M^3\times\{\mathrm{pt}\})
  =
  \Vol_{g_{\hyp}}(M^3).
\]
Thus the constant \(\Vol_{g_{\hyp}}(M^3)\) cannot be improved.

It remains to discuss the equality case.  Suppose the mass minimum of the slice class is equal to the hyperbolic volume, and choose \(T\) as above with
\[
  \mathbf M(T)=\Vol_{g_{\hyp}}(M).
\]
For the component \(\Sigma_{j_0}\) chosen above,
\[
  \mathbf M(T)
  \geq
  \Vol_g(\Sigma_{j_0})
  \geq
  |d|\Vol_{g_{\hyp}}(M)
  \geq
  \Vol_{g_{\hyp}}(M),
\]
so all three inequalities are equalities. Hence \(|d|=1\),
\(\Vol_g(\Sigma_{j_0})=\Vol_{g_{\hyp}}(M)\), and $m_{j_0} = 1$, $\displaystyle T = \Sigma_j m_j \llbracket \Sigma_{j} \rrbracket = \llbracket \Sigma_{j_0} \rrbracket$. That is, no component other than $\Sigma_{j_0}$ can occur with positive multiplicity.  Thus
\[
  T = \llbracket \Sigma \rrbracket 
\]
for a connected smooth embedded $2$-sided hypersurface \(\Sigma = \Sigma_{j_0}\), with
\[
  \deg(p|_\Sigma)=\pm 1.
\]

We next use equality in the stability--Yamabe estimate.  For \(\Sigma\), equality gives
\[
  Y(\Sigma,[g_\Sigma])
  =
  -6\Vol_g(\Sigma)^{2/3}.
\]
Let $\varphi_0 > 0$ be the Yamabe minimizer of $[g_{\Sigma}]$. Then
 \begin{align*}
   -6 \Vol_g(\Sigma)^{2/3}  =  & Y(\Sigma, [g_\Sigma]) \\
   = & Q_{\Sigma}(\varphi_0) \\
   = & \frac{\displaystyle \int_{\Sigma}\left( 8 |\nabla \varphi_0|^2 + \Sc_{\Sigma} \varphi_0^2 \right)}{\displaystyle \left( \int_{\Sigma} |\varphi_0|^6 \right)^{1/3}} \\
   \geq & \frac{\displaystyle 6 \int_{\Sigma} |\nabla \varphi_0|^2 - 6 \int_{\Sigma} \varphi_0^2}{\displaystyle \left( \int_{\Sigma} |\varphi_0|^6 \right)^{1/3}}  \\
   \geq &  - 6 \Vol_g(\Sigma)^{2/3} + \frac{\displaystyle 6\int_{\Sigma}|\nabla \varphi_0|^2}{\displaystyle \left(\int_{\Sigma}|\varphi_0|^6 \right)^{1/3}}. 
 \end{align*}
Hence
 \[ \nabla \varphi_0 \equiv 0, \]
so the minimizer $\varphi_0$ is a constant. By the Yamabe equation, $g_{\Sigma}$ has constant scalar curvature. So $g_{\Sigma}$ is the Yamabe metric of the conformal class $[g_{\Sigma}]$. Since $\varphi_0$ is constant and $\Sc_{\Sigma}$ is constant,
 \[ Q_{\Sigma}(\varphi_0) = \Sc_{\Sigma} \, \Vol_{g}(\Sigma)^{2/3}. \]
Comparing this with
 \[ Y(\Sigma, [g_{\Sigma}]) = - 6 \Vol_g(\Sigma)^{2/3} \]  
gives $\Sc_{\Sigma} \equiv -6$. By Lemma~\ref{thm:yamabe-degree},
 \begin{align*}
  -6 \Vol_{g_{\mathrm{hyp}}}(M)^{2/3} \leq & - 6 \Vol_g(\Sigma)^{2/3} \\
    = & Y(\Sigma, [g_{\Sigma}]) \\
    \leq & \sigma(\Sigma) \\
    \leq & - 6 \Vol_{g_{\mathrm{hyp}}}(M)^{2/3}.
 \end{align*}
 Thus 
  \[ \sigma(\Sigma) = -6 \Vol_{g_{\mathrm{hyp}}}(M)^{2/3} = -6 \Vol_{g}(\Sigma)^{2/3}, \]
  and $g_{\Sigma}$ is the Yamabe metric realizing $\sigma(\Sigma)$. By Schoen \cite{SchoenVariational}, a Yamabe metric realizing the non-positive Yamabe invariant $\sigma(\Sigma)$ is Einstein. In dimension three, every Einstein manifold has constant sectional curvature (see Besse \cite{Besse1987}). Hence \(\Sigma\) is hyperbolic and \(g_\Sigma\) is the hyperbolic Yamabe metric in its conformal class. Since \(\deg(p|_\Sigma)=\pm1\) and
\(\Vol_g(\Sigma)=\Vol_{g_{\hyp}}(M)\), by the equality case of Besson-Courtois-Gallot's volume entropy rigidity theorem (see \cite{BCGMinimalEntropy}), $p|_{\Sigma}$ is homotopic to an isometry. Then Mostow rigidity identifies \((\Sigma,g_\Sigma)\) with \((M,g_{\hyp})\), up to a diffeomorphism of degree \(\pm1\) (see \cite{MostowRigidity}).

The equalities in the stability calculation also force the extrinsic terms to vanish.  Indeed, the stability inequality and the Gauss equation gave
\[
  \int_\Sigma
  \left(
  2|\nabla\varphi|^2+\Sc_\Sigma\varphi^2
  \right)
  \geq
  \int_\Sigma
  \left(
  \Sc_g+|A|^2
  \right)\varphi^2 .
\]
Taking \(\varphi\equiv1\), using \(\Sc_\Sigma\equiv -6\), and using
\(\Sc_g\geq -6\), we obtain
\[
  -6\Vol_g(\Sigma)
  =
  \int_\Sigma \Sc_\Sigma
  \geq
  \int_\Sigma(\Sc_g+|A|^2)
  \geq
  -6\Vol_g(\Sigma)+\int_\Sigma |A|^2 .
\]
Therefore \(A\equiv0\), and the same chain gives
\[
  \Sc_g\equiv -6
  \qquad\text{along }\Sigma .
\]
The Gauss equation then also gives
\[
  \Ric_N(\nu,\nu)=0
  \qquad\text{along }\Sigma .
\]
Thus the Jacobi operator of \(\Sigma\) is the scalar Laplacian:
\[
  L=\Delta_\Sigma+\Ric_N(\nu,\nu)+|A|^2=\Delta_\Sigma .
\]

We use the local area comparison theorem for equality cases of stable hypersurfaces.  The two-dimensional cases go back to Cai--Galloway and Nunes (see \cite{CaiGallowayTori} and \cite{NunesHyperbolic}); the higher-dimensional formulation is due to Moraru (see \cite{MoraruAreaComparison}). We first recall why the required constant mean curvature foliation exists.

Let \(\nu\) be the chosen unit normal along \(\Sigma\), and write nearby hypersurfaces as normal graphs
\[
  \Sigma_w = \{\exp_x(w(x)\nu(x)): x \in \Sigma\}
\]
for \( w \in C^{2,\alpha}(\Sigma) \) small, where $\alpha \in (0, 1)$.  Let \( H_w \) be the mean curvature of \(\Sigma_w \), computed with the graph normal extending \(\nu\).  At \( w = 0 \), the linearization of the mean curvature \(H\) is
\[
  D H_0(\phi)=-L\phi=-\Delta_\Sigma\phi .
\]
The kernel consists exactly of the constant functions.  These constants are the direction in which the foliation parameter moves the initial leaf. To separate this one-dimensional kernel from the elliptic part, define $f: (-\varepsilon, \varepsilon) \times \Sigma \to \R$ by
\[
  f(t, x) = t + w(x), \qquad  t \in (-\varepsilon, \varepsilon),   \qquad  \int_\Sigma w \, \dd \mu_\Sigma = 0,
\]
where $\varepsilon$ is small positive constant. Let
\[
  C^{k,\alpha}_0(\Sigma)
  =
  \left\{\psi\in C^{k,\alpha}(\Sigma):
  \int_\Sigma \psi\,\dd\mu_\Sigma=0\right\}.
\]
If 
 \[ \overline{H}(t, w) = \frac{1}{\Vol_g(\Sigma)} \int_{\Sigma} H_{t+w} \, \dd \mu_{\Sigma} \]
denotes the average value of \(H(t, w) = H_{t+w} \) on \(\Sigma\), consider
\[
  \mathcal{F}(t, w) = H(t, w) - \overline{H}(t, w),
\]
as a map from a neighborhood of \((0,0)\) in
\(\R\times C^{2,\alpha}_0(\Sigma)\) to \(C^{0,\alpha}_0(\Sigma)\). The derivative in the \(w\)-variable at \((0,0)\) is
\[
  D_w \mathcal F_{(0,0)}(\phi)=-\Delta_\Sigma\phi ,
\]
which is an isomorphism
\[
  C^{2,\alpha}_0(\Sigma)\longrightarrow C^{0,\alpha}_0(\Sigma).
\]
The implicit function theorem therefore gives a smooth curve
 \[ w: (-\varepsilon_1, \varepsilon_1) \longrightarrow C_0^{2, \alpha}(\Sigma), \qquad w(t) = w(t, \cdot) \in C_0^{2, \alpha}(\Sigma), \]
 such that $\mathcal{F}(t, w(t)) = 0$ for $t \in (-\varepsilon_1, \varepsilon_1)$, where $\varepsilon_1 \in (0, \varepsilon)$. Since \(D_t\mathcal F_{(0,0)}=0\), uniqueness in the implicit function theorem also gives \( \frac{\partial w}{\partial t}( 0, x) = 0 \). Thus the graph height derivative
\[
  \frac{\partial}{\partial t} \left( t+ w(t, x) \right) = 1+ \frac{\partial w}{\partial t}(t, x)
\]
is positive for \(|t|\) small.  After shrinking \(\varepsilon\), the graphs
are therefore disjoint and form a smooth constant mean curvature foliation
\[
  \{ \Sigma_t \}_{|t| < \varepsilon} = \{ \exp_x ( f(t, x) \nu(x) ) \colon  x \in \Sigma , t \in (-\varepsilon, \varepsilon) \}
\]
near \(\Sigma=\Sigma_0\).  Since the leaves are obtained by a small isotopy
of \(\Sigma\), each \(\Sigma_t\) is homologous to \(\Sigma\).
After precomposing the parametrizations by time dependent diffeomorphisms of $\Sigma$, choose adapted coordinates satisfying $\partial_t = u \nu_t$, where $\nu_t$ is the unit normal on $\Sigma_t$. Then the metric in this tubular neighborhood becomes
\[
  g=u^2 dt^2 + g_t,
\]
where \(u > 0\) is the lapse function, \(g_t\) is the induced metric on \(\Sigma_t\), see \cite{MoraruAreaComparison} for more details. Let \(H(t)\) denote the constant mean curvature of \(\Sigma_t\).  The first
variation formula gives
\[
  \frac{d}{dt}\Vol_g(\Sigma_t)
  = 
  \int_{\Sigma_t}H(t)u\,\dd\mu_t .
\]
Since \(\Sigma=\Sigma_0\) is area minimizing in its homology class and the \(\Sigma_t\)'s are homologous to \(\Sigma\), the function \(t\mapsto \Vol_g(\Sigma_t)\) has a minimum at \(t=0\).

The equalities already proved give the hypotheses of the local area comparison theorem (see \cite[Theorem 4]{MoraruAreaComparison}) for this CMC foliation: \(\Sigma\) is totally geodesic,
\(\Ric_N(\nu,\nu)=0\), \(\Sc_g\equiv -6\) along \(\Sigma\), and
\(g_\Sigma\) attains the Yamabe invariant of \(\Sigma\). After possibly shrinking \(\varepsilon\), the theorem gives
\begin{equation*} 
  \Vol_g(\Sigma_t)\leq \Vol_g(\Sigma_0)
  \qquad (|t|<\varepsilon). 
\end{equation*}
Together with the area minimizing property, this implies
\[
  \Vol_g(\Sigma_t)=\Vol_g(\Sigma_0)
  \qquad (|t|<\varepsilon).
\]
For completeness, we recall the lapse equation argument behind the equality case (see \cite{MoraruAreaComparison} for details). Since $\Vol_g(\Sigma_t) = \Vol_g(\Sigma_0)$, each $\Sigma_t$ is a globally minimizing equality representative for the homology class $[M^3 \times \{ \mathrm{pt}\}] \in H_3(M^3 \times \sph^1; \Z)$. If we repeat the equality arguments proved for $\Sigma$, then it gives
 \[
  A_t\equiv0, \qquad \Sc_g\equiv -6,\qquad \Ric_N(\nu_t,\nu_t) \equiv0  \quad \text{on }\Sigma_t .
\]  
The variation of mean curvature in the normal direction $\partial_t = u \nu_t$ is 
 \[   H'(t)  =  -\Delta_{\Sigma_t}u  - \bigl(\Ric_N(\nu_t,\nu_t)+ |A_t|^2\bigr)u . \]
The global minimizing property implies $H(t) \equiv 0$, hence $H^{\prime}(t) \equiv  0$. Therefore the lapse equation becomes
\[
  \Delta_{\Sigma_t}u=0.
\]
Since \(\Sigma_t\) is closed, \(u\) is constant on each leaf.  After reparametrizing \(t\), we may assume \(u\equiv1\).  Finally,
\[
  \partial_t g_t=2uA_t=0,
\]
since $\Sigma_t$ is totally geodesic. Therefore \(g_t\) is independent of \(t\). Thus there is \(\varepsilon>0\) and a diffeomorphism
\[
  \Phi:\Sigma\times(-\varepsilon,\varepsilon)\longrightarrow U\subset N
\]
with \(\Phi(\Sigma\times\{0\})=\Sigma\), such that
\[
  \Phi^*g=g_\Sigma+dt^2.
\]
Since \(g_\Sigma=g_{\hyp}\), this is the local product splitting near \(\Sigma\).

\medskip

It remains to extend the local splitting globally.  We use the continuation
principle appearing in the equality cases of Bray--Brendle--Neves and Zhu
(compare \cite{BrayBrendleNevesTwoSpheres} and
\cite{ZhuRigiditySpheres}): once the local foliation consists of minimizing
equality leaves, a terminal interior leaf is again an equality leaf, so the
local splitting theorem can be restarted there.

Cut \(N=M^3\times\sph^1\) open along \(\Sigma\).  Because \(\Sigma\)
represents the class dual to the \(\sph^1\)-factor and \(p|_\Sigma\) has
degree \(\pm1\), after the above identification the cut-open manifold is a
connected cobordism from \(M\) to \(M\).  Denote this cut-open manifold by
\(W\), and denote its boundary components by \(\Sigma_-\) and \(\Sigma_+\),
where \(\Sigma_-\) is the initial copy of \(\Sigma\). Let \(I\subset[0,\infty)\) be the maximal interval on which the
inward normal exponential map from \(\Sigma_-\) gives an embedded product
region
\[
  F:\Sigma\times I\longrightarrow W,\qquad F^*g=g_{\hyp}+dt^2,
\]
with \(F(\Sigma,0)=\Sigma_-\).  In this product region the second coordinate
is the normal distance from \(\Sigma_-\).  Thus the picture is literally a
cylinder whose horizontal slices are
\[
  \Sigma_t=F(\Sigma\times\{t\})
\]
and whose vertical curves \(t\mapsto F(x,t)\) are unit-speed normal geodesics.
Equivalently, before the product region fills \(W\), each \(\Sigma_t\) is the
graph of the constant function \(t\) over the initial slice \(\Sigma_-\).

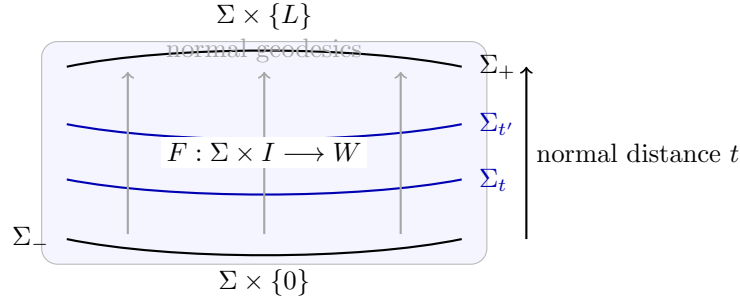
\begin{center}
\begin{tikzpicture}[scale=0.95, every node/.style={font=\small}]
  \fill[blue!4, rounded corners=6pt] (0,0) rectangle (6.2,3.1);
  \draw[gray!55, rounded corners=6pt] (0,0) rectangle (6.2,3.1);

  \draw[thick]
    (0.35,0.35) .. controls (1.8,0.05) and (4.4,0.05) .. (5.85,0.35);
  \draw[thick]
    (0.35,2.75) .. controls (1.8,3.05) and (4.4,3.05) .. (5.85,2.75);

  \draw[blue!70!black,thick]
    (0.35,1.18) .. controls (1.8,0.90) and (4.4,0.90) .. (5.85,1.18);
  \draw[blue!70!black,thick]
    (0.35,1.95) .. controls (1.8,1.67) and (4.4,1.67) .. (5.85,1.95);

  \node[left] at (0.25,0.35) {\(\Sigma_-\)};
  \node[right] at (5.95,2.75) {\(\Sigma_+\)};
  \node[blue!70!black,right] at (5.95,1.18) {\(\Sigma_t\)};
  \node[blue!70!black,right] at (5.95,1.95) {\(\Sigma_{t'}\)};

  \foreach \x in {1.2,3.1,5.0} {
    \draw[gray!70,->,thick] (\x,0.42) -- (\x,2.68);
  }
  \node[gray!70] at (3.1,2.98) {normal geodesics};

  \draw[->,thick] (6.75,0.35) -- (6.75,2.75);
  \node[right] at (6.75,1.55) {normal distance \(t\)};

  \node at (3.1,-0.25) {\(\Sigma\times\{0\}\)};
  \node at (3.1,3.45) {\(\Sigma\times\{L\}\)};
  \node[fill=white,inner sep=2pt] at (3.1,1.55)
    {\(F:\Sigma\times I\longrightarrow W\)};
\end{tikzpicture}
\end{center}

Consequently each leaf separates \(\Sigma_-\) from \(\Sigma_+\) and represents the same slice class.  The interval is nontrivial because the local splitting theorem gives a product collar of \(\Sigma_-\).  Moreover, the same theorem gives the continuation property at every interior leaf: if \(t_0\in I\) and \(\Sigma_{t_0}\) is still in the interior of \(W\), then \(\Sigma_{t_0}\) is a minimizing equality representative of the slice class, so the local product splitting at \(\Sigma_{t_0}\) extends the product coordinate a little farther in the normal direction. Thus a maximal product interval can fail to continue only by reaching a limiting leaf. We now rule
out any limiting leaf in the interior of \(W\).

Indeed, suppose that \(t_i\in I\) and \(t_i\to t_*\). On the product region the second fundamental forms of the leaves vanish and the induced metrics are all equal to \(g_{\hyp}\). The normal geodesics therefore have unit speed and uniform geometry up to time \(t_*\).  If they have not reached
\(\Sigma_+\), the maps \(F(\cdot,t_i)\) have a smooth limiting immersion
\[
  \Sigma_{t_*}=F(\Sigma\times\{t_*\})
\]
with induced metric \(g_{\hyp}\).  We now interpret this limit as an integral current.  Since the slices 
\( F(\cdot, t_i)_{\#}(\llbracket \Sigma \rrbracket) \) are all homologous to \([\Sigma]\), the flat limit \(F(\cdot,t_*)_\#( \llbracket \Sigma \rrbracket ) \) represents the same slice class.  Its mass is at most \(\Vol_{g_{\hyp}}(M)\).  On the other hand, \(\llbracket \Sigma \rrbracket \) is the mass minimizing current in the slice class and, in the equality case,
\[
  \mathbf M( \llbracket \Sigma \rrbracket )=\Vol_g(\Sigma)=\Vol_{g_{\hyp}}(M).
\]
Thus the limiting current is again mass minimizing and has exactly this mass. By regularity for codimension-one mass minimizers in dimension \(4\), it is represented by a $2$-sided smooth embedded hypersurface, possibly with integer multiplicities. Applying the equality discussion above to this embedded hypersurface rules out extra components and multiplicities.  We continue to denote the resulting connected multiplicity-one hypersurface by \(\Sigma_{t_*}\). It is a minimizing equality representative of the slice class and satisfies
\[
  \Vol_g(\Sigma_{t_*})=\Vol_{g_{\hyp}}(M).
\]
Further, equality in the above area formulas excludes cancellation. Hence the image $F(\Sigma, t_{*})$ is exactly the embedded support of the limiting current $\llbracket \Sigma_{t_{*}} \rrbracket$. Therefore $F(\cdot, t_{*}): \Sigma \to \Sigma_{t_{*}}$ is a finite Riemannian covering map. If it is a $k$-sheeted covering, then
 \[ \Vol_{\mathrm{hyp}}(\Sigma) = k \Vol_{g}(\Sigma_{t_{*}}). \]
But both volumes are equal to $\Vol_{\mathrm{hyp}}(M)$. Then $k = 1$ and $F(\cdot , t_{*})$ is an embedding. Moreover, $t_{*} < \infty$, since for any embedded product region $F(\Sigma \times [0, t]) \subset W$,
 \[ 
 t \, \Vol_{\mathrm{hyp}}(M) = \Vol_g(F(\Sigma \times [0, t])) \leq \Vol_g(W). 
 \]
We claim that $\Sigma_{t_{*}}$ is disjoint from every earlier hypersurface $\Sigma_{t}$. Suppose otherwise, choose $\bar{t} < t_{*}$ such that $\Sigma_{\bar{t}} \cap \Sigma_{t_{*}} \neq \emptyset$. Define 
 \[ t_0 = \min\{0 \leq t \leq \bar{t} : \Sigma_t \cap \Sigma_{t_{*}} \neq \emptyset \}. \] 
The contact time set is nonempty and closed, by compactness of $\Sigma$ and continuity of $F$. Hence its minimum exists. 
Note that the connected region $F(\Sigma \times (t_0, t_{*}))$ is disjoint from $\Sigma_{t_0}$. So it lies in a single component $C^{+}_{t_0}$ of $W \setminus \Sigma_{t_0}$. Consequently,
 \[ \Sigma_{t_{*}} \subset \overline{C^{+}_{t_0}}. \] 
Thus, at any intersection with $\Sigma_{t_0}$, the limiting hypersurface $\Sigma_{t_{*}}$ is tangent to $\Sigma_{t_0}$ and lies locally on one side of it. The strong maximum principle therefore gives $\Sigma_{t_{*}} = \Sigma_{t_0}$.
At this common hypersurface, the two normals are opposite. And note that $C^{+}_{t_0}$ is the positive side of $\Sigma_{t_0}$ but the negative side of $\Sigma_{t_{*}}$. Therefore, the uniqueness of normal coordinate gives $\Sigma_{t_{*}-\delta} = \Sigma_{t_{0}+\delta}$. Choosing $0 < \delta < (t_{*} - t_0)/2$ contradicts injectivity of $F$ on $\Sigma \times [0, t_{*})$. Hence $\Sigma_{t_*}$ is disjoint from all the earlier embedded hypersurfaces $\Sigma_t$, for $0 \leq t < t_{*}$. If \(\Sigma_{t_*}\) lies in the interior of \(W\), the local splitting theorem applied to \(\Sigma_{t_*}\), with the normal chosen to continue the previous
normal coordinate, extends the product region past \(t_*\).  Thus no interior endpoint can occur.

Consequently the only possible endpoint of the maximal interval is the other boundary component \(\Sigma_+\).  Since \(W\) is compact and connected, the swept-out region cannot stop before reaching \(\Sigma_+\). If the limiting leaf $\Sigma_{t_{*}}$ touches the boundary \( \Sigma_{+} \), then it must coincide with $\Sigma_{+}$. Since both \( \Sigma_{t_{*}} \) and $\Sigma_{+}$ are closed minimal hypersurfaces, and $\Sigma_{+}$ has a product collar from the local splitting at $\Sigma$, the strong maximum principle leads to the coincidence. Set $L = t_{*}$. The image $F(\Sigma \times [0, L])$ is closed, locally open in $W$ by the product collars, and nonempty. Then the connectedness of $W$ gives
 \[ W = F(\Sigma \times [0, L]). \] 
Hence, 
\[
  W\cong \Sigma\times [0,L],  \qquad  g=g_{\hyp}+dt^2
\]
with \(\Sigma\times\{0\}=\Sigma_-\) and \(\Sigma\times\{L\}=\Sigma_+\).

It remains only to identify the gluing.  Gluing the two boundary components
of \(W\) gives the mapping torus
 \[ N_{\varphi} = \Sigma \times [0, L] / (x, 0) \equiv (\varphi (x), L) \]
 of an isometry
\[
  \varphi:(\Sigma,g_{\hyp})\longrightarrow(\Sigma,g_{\hyp}),
\]
with metric \(g_{\hyp}+dt^2\) on the cylinder before gluing.  The fibration class of this mapping torus $N_{\varphi}$ is the cohomology class dual to \(\Sigma\), which is the original \(\sph^1\)-class on \(M\times\sph^1\). Hence the kernel of the induced homomorphism
\[
  \pi_1(N_{\varphi}) \longrightarrow \Z
\]
is the subgroup \(\pi_1(\Sigma)\) coming from the product. On the other hand, for the mapping torus $N_{\varphi}$ this kernel is the fiber subgroup \(\pi_1(\Sigma) \cong \pi_1(M)\). Then the fundamental group of the mapping torus fits into
 \[ 1 \longrightarrow \pi_1(\Sigma) \longrightarrow \pi_1(N_{\varphi}) \xlongrightarrow{q} \Z \longrightarrow 1.  \]
 Choose a stable letter $t \in \pi_1(N_{\varphi})$ with $q(t) = 1$. Then
 \[ \pi_1(N_{\varphi}) \cong \pi_1(\Sigma) \rtimes_{\varphi_{*}} \Z, \quad t g t^{-1} = \varphi_{*}(g). \]
After using the preceding isometric identification \( (\Sigma, g_{\mathrm{hyp}}) \cong (M, g_{\mathrm{hyp}}) \), the original product structure gives an isomorphism  
 \[ \Theta\colon \pi_1(N_{\varphi}) \longrightarrow \pi_1(\Sigma) \times \Z \]
 satisfying $\mathrm{proj}_2 \circ \Theta = q$. Let 
 \[ \beta = \mathrm{proj}_1 \circ \Theta|_{\pi_1(\Sigma)}\colon \pi_1(\Sigma) \longrightarrow \pi_1(\Sigma) \]
be the isomorphism. Writing $\Theta(t) = (a, 1)$ for some $a \in \pi_1(\Sigma)$, we obtain
 \[ \beta \circ \varphi_{*} \circ \beta^{-1} = \mathrm{Inn}_{a}. \]
Thus $[\varphi_{*}]$ is trivial in $\mathrm{Out}(\pi_1(\Sigma))$. Since $\Sigma$ is aspherical, $\varphi$ is homotopic to the identity. By Mostow rigidity theorem, the homomorphism 
 \[ \mathrm{Isom}(\Sigma , g_{\mathrm{hyp}}) \longrightarrow \mathrm{Out}(\pi_1(\Sigma)) \]
is injective, and therefore $\varphi = \mathrm{id}$. Thus the mapping torus is the ordinary product, and, after a diffeomorphism preserving the slice class, the metric is
\[
  g_{\hyp}+h_{\sph^1}.
\]
This proves the rigidity statement.
\end{proof}

\bigskip

\section{The Higher-Dimensional Problem}\label{sec:higher-dimensional}

The proof in dimension three has two distinct parts.  The construction of a mass minimizing hypersurface and the estimate obtained from stability extend to higher dimensions whenever the minimizer is smooth.  The sharp estimate
for the Yamabe invariant under a map of nonzero degree is presently available here only in dimension three.  We formulate the higher-dimensional problem so that these two parts remain separate.

\subsection{The volume and rigidity questions}

Let \((M^n,g_{\hyp})\) be a closed oriented hyperbolic \(n\)-manifold with \(\secop_{g_{\hyp}}\equiv-1\), and let \(N=M\times\sph^1\). Set
\[
  \alpha=[M\times\{\mathrm{pt}\}]\in H_n(N;\Z).
\]
For a Riemannian metric \(g\) on \(N\), denote by
\[
  \mathcal A_g(\alpha) = \inf\bigl\{\mathbf M_g(T):
  T \text{ is an integral }n\text{-cycle representing }\alpha\bigr\}
\]
the least mass of the slice class.

\begin{question}[Volume comparison]\label{q:higher-dimensional-volume}
If \(g\) is a smooth Riemannian metric on \(N\) satisfying
\[
  \Sc_g\geq-n(n-1),
\]
is it true that
\[
  \mathcal A_g(\alpha)\geq\Vol_{g_{\hyp}}(M)?
\]
\end{question}

\begin{question}[Rigidity]\label{q:higher-dimensional-rigidity}
Under the hypotheses of Question~\ref{q:higher-dimensional-volume}, suppose
that
\[
  \mathcal A_g(\alpha)=\Vol_{g_{\hyp}}(M).
\]
Must \(g\), after a diffeomorphism preserving \(\alpha\), be of the form
\[
  g_{\hyp}+h_{\sph^1}
\]
for some metric \(h_{\sph^1}\) on \(\sph^1\)?
\end{question}

The current formulation is useful in dimensions \(n\geq7\), where a mass minimizing representative may have a singular set.  For \(3 \leq n \leq6 \), regularity gives a smooth embedded $2$-sided minimizing hypersurface, possibly with several components and integer multiplicities.

\subsection{Relation with Schoen's hyperbolic volume conjecture}

Question~\ref{q:higher-dimensional-volume} is stronger than the volume inequality in Conjecture~\ref{conj:schoen-volume}. Let \(g_M\) be a metric on \(M\) satisfying
\[
  \Sc_{g_M}\geq-n(n-1).
\]
For any metric \(h_{\sph^1}\) on \(\sph^1\), the product metric \(g_M+h_{\sph^1}\) has the same scalar curvature, and the slice \(M\times\{\mathrm{pt}\}\) represents \(\alpha\).  Therefore
\[
  \mathcal A_{g_M+h_{\sph^1}}(\alpha)  \leq  \Vol_{g_M}(M).
\]
An affirmative answer to Question~\ref{q:higher-dimensional-volume} would then give
\[
  \Vol_{g_M}(M)\geq\Vol_{g_{\hyp}}(M),
\]
which is Schoen's conjectured inequality. We do not know whether the converse holds. For $3 \leq n \leq 6$, Conjecture~\ref{conj:degree-yamabe} supplies the missing sharp estimate for the hypersurface approach to Question~\ref{q:higher-dimensional-volume} described below.

\subsection{The standard hypersurface reduction}

Assume \(3 \leq n\leq6\), and let \(T\) minimize mass in the class \(\alpha\).
The regularity theory for codimension-one mass minimizers writes
\[
  T=\sum_j m_j \llbracket \Sigma_j \rrbracket,
\]
where the \(\Sigma_j\) are smooth, connected, oriented, stable minimal
hypersurfaces and \(m_j\in\N\).  Let \(p:N\to M\) be the projection.  Since \( [p_\#T] =[M]\),
\[
  \sum_jm_j\deg(p|_{\Sigma_j})=1.
\]
At least one component, denoted by \(\Sigma_0\), therefore has nonzero degree $d=\deg(p|_{\Sigma_0})$.

The stability inequality and the Gauss equation give, for every smooth function \(\varphi\) on \(\Sigma_0\),
\[
  \int_{\Sigma_0}
  \left(2|\nabla\varphi|^2+\Sc_{\Sigma_0}\varphi^2\right)
  \geq
  -n(n-1)\int_{\Sigma_0}\varphi^2.
\]
Since \(4(n-1)/(n-2)>2\), the Yamabe quotient and H\"older's inequality yield the standard estimate
\[
  Y(\Sigma_0,[g_{\Sigma_0}])
  \geq
  -n(n-1)\Vol_g(\Sigma_0)^{2/n}.
\]
Thus regularity, stability, the Gauss equation, and the existence of a component of nonzero degree do not constitute the main unresolved step.

\subsection{The comparison conjecture}

The missing estimate is the following degree form of the Yamabe comparison.

\begin{conjecture}[Degree Yamabe comparison]\label{conj:degree-yamabe}
Let \((M^n,g_{\hyp})\) be a closed oriented $n$-dimensional hyperbolic manifold with \(\secop_{g_{\hyp}}\equiv-1\).  If \(X^n\) is a closed oriented $n$-dimensional manifold and \(f:X\to M\) has degree \(d\neq0\), then every conformal class \([h]\) on
\(X\) satisfies
\[
  Y(X,[h]) \leq -n(n-1) \left(|d|\Vol_{g_{\hyp}}(M)\right)^{2/n}.
\]
\end{conjecture}

In dimension three, Lemma~\ref{thm:yamabe-degree} proves this statement by combining geometrization with simplicial volume.  No corresponding formula for the Yamabe invariant is known in higher dimensions.  The Besson--Courtois--Gallot theorem proves the analogous sharp degree inequality for minimal entropy (see \cite{BCGMinimalEntropy}), but it does not give the scalar-curvature estimate in Conjecture~\ref{conj:degree-yamabe}.

The conjecture gives a direct route to the volume comparison.

\begin{proposition}\label{prop:degree-yamabe-reduction}
Let \(3\leq n\leq6\).  If Conjecture~\ref{conj:degree-yamabe} holds for the hyperbolic manifold \(M\), then Question~\ref{q:higher-dimensional-volume}
has an affirmative answer for \(M\).
\end{proposition}

\begin{proof}
Apply Conjecture~\ref{conj:degree-yamabe} to
\(p|_{\Sigma_0}:\Sigma_0\to M\).  Together with the stability estimate, it gives
\[
  -n(n-1)\Vol_g(\Sigma_0)^{2/n}
  \leq
  Y(\Sigma_0,[g_{\Sigma_0}])
  \leq
  -n(n-1)
  \left(|d|\Vol_{g_{\hyp}}(M)\right)^{2/n}.
\]
Consequently,
\[
  \Vol_g(\Sigma_0)
  \geq
  |d|\Vol_{g_{\hyp}}(M)
  \geq
  \Vol_{g_{\hyp}}(M).
\]
Since \(\mathbf M_g(T)\geq\Vol_g(\Sigma_0)\), the same lower bound holds for \(\mathcal A_g(\alpha)\).
\end{proof}

For a Yamabe metric \(\hat h\in[h]\) normalized by
\(\Sc_{\hat h}\equiv-n(n-1)\), Conjecture~\ref{conj:degree-yamabe} becomes the volume inequality
\[
  \Vol_{\hat h}(X)
  \geq
  |d|\Vol_{g_{\hyp}}(M).
\]
It is therefore a degree theorem for all manifolds mapping to \(M\), not only a comparison between two metrics on \(M\).

\subsection{Rigidity in higher dimensions}

An affirmative answer to Question~\ref{q:higher-dimensional-volume} implies the volume inequality part of Conjecture~\ref{conj:schoen-volume}, and hence by the equivalence recalled in \cite{YuanVolumeComparison}, implies Conjecture~\ref{conj:schoen-yamabe}. For $3 \leq n \leq 6$, Conjecture~\ref{conj:degree-yamabe} supplies the sharp bound for the hypersurface method above. Then completing the rigidity argument would additionally require analyzing equality throughout and carrying out the corresponding local-to-global splitting argument. For $n \geq 7$, we must also treat the singular set of a mass minimizing hypersurface.

\bibliographystyle{amsplain}
\bibliography{hyperbolic_volume_bound}

\end{document}